\documentclass[11pt]{amsart}
\usepackage{amsmath}
\usepackage[english,  activeacute]{babel}
\usepackage[latin1]{inputenc}
\usepackage{amssymb}
\usepackage{amsthm}
\usepackage{graphics,graphicx}
\usepackage{array}
\usepackage{a4wide}
\allowdisplaybreaks
\usepackage{color, url}
\usepackage{float}

\theoremstyle{plain}
\newtheorem{theorem}{Theorem}

\newtheorem{corollary}[theorem]{Corollary}
\newtheorem{lemma}[theorem]{Lemma}
\theoremstyle{definition}

\date{\today}

\title[Enumeration of non-crossing partitions]{Enumeration of non-crossing partitions according to subwords with repeated letters}

\begin{document}

\author[M. Shattuck]{Mark Shattuck}
\address{Department of Mathematics, University of Tennessee,
37996 Knoxville, TN}
\email{mark.shattuck2@gmail.com}

\begin{abstract}
An avoidance pattern where the letters within an occurrence of which are required to be adjacent is referred to as a \emph{subword}.  In this paper, we enumerate members of the set $NC_n$ of non-crossing partitions of length $n$ according to the number of occurrences of several infinite families of subword patterns each containing repeated letters.  As a consequence of our results, we obtain explicit generating function formulas counting the members of $NC_n$ for $n \geq 0$ according to all subword patterns of length three containing a repeated letter.  Further, simple expressions are deduced for the total number of occurrences over all members of $NC_n$ for the various families of patterns.  Finally, combinatorial proofs can be given explaining three infinite families of subword equivalences over $NC_n$, which generalize the following equivalences:  $211 \equiv 221$, $1211\equiv 1121$ and $112 \equiv 122$.
\end{abstract}

\maketitle

\section{Introduction}

A collection of disjoint nonempty subsets of a set whose union is the set is known as a \emph{partition}, with the constituent subsets referred to as \emph{blocks} of the partition.  Let $[n]=\{1,2,\ldots,n\}$ for $n \geq 1$, with $[0]=\varnothing$.  The set of partitions of $[n]$ containing exactly $k$ blocks will be denoted by $\mathcal{P}_{n,k}$, with $\mathcal{P}_n=\cup_{k=0}^n\mathcal{P}_{n,k}$ denoting the set of all partitions of $[n]$.  A partition $\Pi=B_1/B_2/\cdots/B_k \in \mathcal{P}_{n,k}$ is said to be in \emph{standard form} if its blocks $B_i$ are such that $\min(B_i)<\min(B_{i+1})$ for $1 \leq i \leq k-1$.  A partition $\Pi$ in standard form can be represented sequentially by writing $\pi=\pi_1\cdots\pi_n$, where $i \in B_{\pi_i}$ for each $i \in [n]$ (see, e.g., \cite{Man}). The sequence $\pi$ is referred to as the \emph{canonical sequential form} of the partition $\Pi$.  Then $\Pi$ in standard form implies $\pi_{i+1} \leq \max(\pi_1\cdots \pi_i)+1$ for $1 \leq i \leq n-1$, which is known as the \emph{restricted growth} condition (see, e.g., \cite{Mi}).

A partition $\Pi$ is said to be \emph{non-crossing} \cite{Kl} if its sequential representation $\pi$ contains no subsequence of the form $a$-$b$-$a$-$b$, where $a<b$ (i.e., if $\pi$ avoids the pattern 1-2-1-2 in the classical sense).  Let $NC_n$ denote the set of non-crossing partitions of $[n]$; recall that $|NC_n|=C_n$ for all $n \geq 0$, where $C_n=\frac{1}{n+1}\binom{2n}{n}$ is the $n$-th Catalan number.  We will denote the Catalan number generating function $\sum_{n \geq 0}C_nx^n=\frac{1-\sqrt{1-4x}}{2x}$ by $C(x)$.

Let $\tau=\tau_1\cdots \tau_m$ be a sequence of positive integers whose set of distinct letters comprise $[\ell]$ for some $1 \leq \ell \leq m$.  Then the sequence $\rho=\rho_1\cdots \rho_n$ is said to \emph{contain} $\tau$ as a \emph{subword} (pattern) if some string of consecutive letters of $\rho$ is order-isomorphic to $\tau$.  That is, there exists an index $i \in [n-m+1]$ such that $\rho_i\rho_{i+1}\cdots\rho_{i+m-1}$ is isomorphic to $\tau$.  If no such index $i$ exists, then $\rho$ \emph{avoids} $\tau$ as a subword.

Here, we will be interested in counting the members of $NC_n$ according to the number of occurrences of certain subword patterns, focusing on several infinite families of patterns.  Let $\mu_\tau(\pi)$ denote the number of occurrences of the subword $\tau$ in the partition $\pi$.  We compute the generating function $F=F_\tau$ for the distribution of $\tau$ on $NC_n$ where
$$F=\sum_{n\geq0}\left(\sum_{\pi \in NC_n}q^{\mu_\tau(\pi)}\right)x^n$$
in several cases when $\tau$ has one or more repeated letters.  This extends recent work initiated in \cite{MSh} which focused on subwords where all of the letters in a pattern were distinct.  We remark that other finite discrete structures with sequential representations that have been enumerated according to the number of subwords include $k$-ary words \cite{BM}, set partitions \cite{MSY} and involutions \cite{MS1}.  For examples of other types of statistics which have been studied on non-crossing partitions, we refer the reader to \cite{LF,MS2,Si,YY,ZZ}.

This paper is organized as follows.  In the next section, we enumerate members of $NC_n$ according to four infinite families of subword patterns and compute the corresponding generating function $F$ in each case.  Simple formulas for the total number of occurrences on $NC_n$ for the various patterns are deduced from our formulas for $F$. Further, an explicit bijection is defined which demonstrates the equivalence of the subwords $(\rho+1)1^a$ and $(\rho'+1)1^{a'}$ of the same length.  In the third section, the pattern $12\cdots (m-1)m^a$ is treated using the \emph{kernel method} \cite{HM} and a formula for the generating function of its joint distribution with an auxiliary parameter on $NC_n$ is found.  Finally, a bijection is given which demonstrates the equivalence of $1^a23\cdots m$ and $12\cdots (m-1)m^a$ as subwords on $NC_n$ for all $a,m \geq 2$.

As special cases of our results, we obtain $F_\tau$  for all $\tau$ of length three containing a repeated letter.  See Table \ref{tab1} below, where the equation satisfied by $F_\tau$ is given for each $\tau$.  Note that the case $212$ is trivial since any partition $\pi$ containing a string $x$ of the form $x=bab$ where $a<b$ must contain an occurrence of 1-2-1-2, upon considering the leftmost occurrence of the letter $a$ in $\pi$ together with $x$.

\begin{table}[htp]
\begin{center}
\begin{tabular}{|l|l|l|}
  \hline
  Subword & Generating function equation & Reference \\\hline\hline
$111$ & $x(1-qx+(q-1)x^2)F^2=(1-qx+(q-1)x^3)(F-1)$&Corollary \ref{1a1b2cor1}\\\hline
$112$ & $x(1+(q-1)x)F^2=(1+(q-1)x^2)F-1$  &Corollary \ref{1a1b2cor1} \\\hline
$121$ & $xF^2=(1-(q-1)x^2)(F-1)$&Theorem \ref{1arho1b} \\\hline
$122$ & $x(1+(q-1)x)F^2=(1+(q-1)x^2)F-1$ &Theorem \ref{t1m-1mal2} \\\hline
$211$ & $x(1+(q-1)x)F^2=(1+2(q-1)x^2)F-1-(q-1)x^2$ &Theorem \ref{thtau1b} \\\hline
$212$ & $xF^2=F-1$&Trivial\\\hline
$221$ & $x(1+(q-1)x)F^2=(1+2(q-1)x^2)F-1-(q-1)x^2$ &Theorem \ref{thtau1b} \\\hline
\end{tabular}\vspace{10px}
\caption{Generating functions $F=F_\tau$ for $\tau$ of length three containing a repeated letter}\label{tab1}
\end{center}
\end{table}

\section{Distributions of some infinite families of patterns}

We first consider the patterns $\tau=1^a$ and $\rho=1^b2$, where $a, b \geq 1$, and treat them together as a joint distribution on $NC_n$.  We shall determine a formula for the generating function (gf) of this distribution given by
$$\sum_{n\geq0}\left(\sum_{\pi \in NC_n}p^{\mu_\tau(\pi)}q^{\mu_\rho(\pi)}\right)x^n,$$
which we will denote by $F$.  We will make use of the \emph{symbolic} enumeration method (see, e.g., \cite{FS}) in finding $F$.

\begin{theorem}\label{th1a1b2}
If $a \geq b \geq 1$, then the generating function $F$ enumerating the members of $NC_n$ for $n \geq 0$ jointly according to the number of occurrences of $1^a$ and $1^b2$ satisfies
\footnotesize{\begin{equation}\label{th1a1b2e1}
(x-px^2+q(p-1)x^a+(q-1)(1-px)x^b)F^2=(1-px+q(p-1)x^a+(q-1)(1-px)x^b)F-1+px-(p-1)x^a.
\end{equation}}\normalsize
If $1\leq a<b$, then $F$ satisfies
\footnotesize{\begin{equation}\label{th1a1b2e2}
(x-px^2+(p-1)x^a+(q-1)(1-x)p^{b-a+1}x^b)F^2=(1-px+(p-1)x^a+(q-1)(1-x)p^{b-a+1}x^b)F-1+px-(p-1)x^a.
\end{equation}}
\end{theorem}
\begin{proof}
First assume $a \geq b \geq 1$ and consider the following cases on $\pi \in NC_n$:
(i) $\pi=1^n$ for some $0 \leq n \leq a-1$, (ii) $\pi=1^n$, where $n \geq a$, (iii) $\pi=1^r\alpha\beta$, where $1 \leq r \leq b-1$, $\alpha$ is nonempty and contains no $1$'s and $\beta$ starts with $1$ if nonempty, (iv) $\pi$ as in (iii), but where $b \leq r \leq a-1$, or (v) $\pi$ as in (iii), but where $r \geq a$.  Combining cases (i)--(v) implies $F$ is determined by
$$F=\frac{1-x^a}{1-x}+\frac{px^a}{1-px}+\left(\frac{x-x^b}{1-x}+q\frac{x^b-x^a}{1-x}+pq\frac{x^a}{1-px}\right)F(F-1).$$
Note that the sections $\alpha$ and $\beta$ of $\pi$ are determined by the factors $F-1$ and $F$, respectively, in cases (iii)--(v).  Further, $r \geq b$ in (iv) and (v) implies that there is an extra occurrence of $\rho$ (accounted for by the lone $q$ factor) arising due to the initial run of $1$'s within $\pi$ and the first letter of $\alpha$.  After simplification, the preceding equation for $F$ rearranges to give \eqref{th1a1b2e1}.

If $1\leq a<b$, then by similar reasoning we have that $F$ satisfies
$$F=\frac{1-x^a}{1-x}+\frac{px^a}{1-px}+\left(\frac{x-x^a}{1-x}+\frac{px^a-p^{b-a+1}x^b}{1-px}+\frac{p^{b-a+1}qx^b}{1-px}\right)F(F-1),$$
which simplifies to gives \eqref{th1a1b2e2}.
\end{proof}

Taking $q=1$ and $p=1$ in Theorem \ref{th1a1b2}, and solving for $F$ (replacing $p$ by $q$ in the resulting formula in the first case), yields the following result.
\begin{corollary}\label{1a1b2cor1}
The generating functions counting members of $NC_n$ for $n \geq 0$ according to the number of occurrences of the patterns $1^m$ and $1^m2$ where $m \geq 1$ are given respectively by
$$\frac{1-qx+(q-1)x^m-\sqrt{(1-qx+(q-1)x^m)((1-4x)(1-qx)-3(q-1)x^m)}}{2x(1-qx+(q-1)x^{m-1})}$$
and
$$\frac{1+(q-1)x^m-\sqrt{(1-(q-1)x^m)^2-4x}}{2x(1+(q-1)x^{m-1})}.$$
\end{corollary}

Differentiating the formulas in Corollary \ref{1a1b2cor1} with respect to $q$, and extracting the coefficient of $x^n$, yields simple expressions for the total number of occurrences of the respective subwords on $NC_n$.
\begin{corollary}\label{1a1b2cor2}
The total number of occurrences of $1^m$ and $1^m2$ within all the members of $NC_n$ for $n \geq m\geq 1$ are given by $\binom{2r}{r+1}$ and $\binom{2r-1}{r+1}$, respectively, where $r=n-m+1$.
\end{corollary}
\begin{proof}
It is also possible to provide a combinatorial explanation of these formulas.  For the first,  suppose that there is a letter $x$ in the $i$-th position within a member of $NC_r$, where $1 \leq i \leq r$.  Then insert $m-1$ additional copies of $x$ to directly follow the one already present in position $i$ and mark the occurrence of the subword $1^m$ in the resulting member of $NC_n$.  Note that all occurrences of $1^m$ within members of $NC_n$ arise uniquely in this way, which yields $rC_r=\binom{2r}{r+1}$ total occurrences. For the second formula, consider an ascent $xy$ in $\pi \in NC_r$ and insert $m-1$ additional copies of $x$ between $x$ and $y$.  This results in an occurrence of $\tau=1^m2$ within a member of $NC_n$ in which the role of the `$2$' is played by $y$.  Thus, counting occurrences of $\tau$ in $NC_n$ is equivalent to counting ascents in $NC_r$.  Note that the number of ascents in a non-crossing partition $\pi$ equals $\mu(\pi)-1$ for all $\pi$, where $\mu(\pi)$ denotes the number of blocks of $\pi$.  Since $\mu$ has a Narayana distribution on $NC_r$ (see, e.g., \cite[A001263]{Sloane}), it follows that the number of blocks over $NC_r$ equals $\sum_{i=1}^{r}\frac{i}{r}\binom{r}{i}\binom{r}{i-1}=\binom{2r-1}{r}$.  Since the number of ascents is always one less than the number of blocks, we have that the total number of occurrences of $\tau$ in all the members of $NC_n$ is given by $\binom{2r-1}{r}-C_r=\binom{2r-1}{r+1}$.
\end{proof}

Given a sequence $\rho$ and a number $x$, let $\rho+x$ denote the sequence obtained by adding $x$ to each entry of $\rho$.  Let $\tau=(\rho+1)1^b$, where $\rho$ is a sequential representation of a non-crossing partition of length $a \geq 1$.  Assume further that $\rho$ starts with a single 1 if $b \geq 2$ (with no such restriction if $b=1$).  Then we have the following general formula for $F_\tau$.

\begin{theorem}\label{thtau1b}
Let $\tau=(\rho+1)1^b$, where $\rho$ is of length $a \geq 1$ as described and $b \geq 1$.  Then the generating function counting the members of $NC_n$ for $n \geq0$ according to the number of occurrences of $\tau$ is given by
$$\frac{1+2(q-1)x^{a+b-1}-\sqrt{1-4x-4(q-1)x^{a+b}}}{2x(1+(q-1)x^{a+b-2})}.$$
\end{theorem}
\begin{proof}
Let $G=G_\tau$ be the gf that enumerates $\pi \in NC_n$ for $n \geq 0$ according to the number of occurrences of $\tau$ in $\pi0^b$ and $F=F_\tau$ denote the usual gf.  We first establish the relation
\begin{equation}\label{thtau1be1}
G=F+(q-1)x^{a-1}(F-1).
\end{equation}
To do so, first note that $F$ and $G$ assign the same $q$-weights to non-crossing partitions except for those of the form $\pi=\alpha\beta$, where $\beta$ corresponds to an occurrence of the subword $\rho$.  We now describe how such partitions can be formed.  Let $x$ denote the first letter of $\beta$.  Let $\rho=\rho_1\rho_2\cdots \rho_a$ and $\rho'=\rho_2\cdots \rho_a$.  Let $\rho^*$ be the sequence obtained from $\rho'$ by replacing each $1$ in $\rho'$ with $x$ and each letter $i>1$  with $i+m-1$, where $m=\max(\alpha \cup\{x\})$. Then appending $\rho^*$ to the partition $\alpha x$ gives $\pi$ of the form stated above, with $\alpha x$ representing an arbitrary  member of $NC_n$ for some $n \geq 1$.  Further, since $\rho$ starts with a single $1$ if $b>1$, we have that appending $\rho^*$ as described to $\alpha x$ does not introduce an occurrence of $\tau$ involving the last letters of $\alpha x$ and the first of $\rho^*$ (as $\rho^*$ must start with $m+1$ if nonempty when $b>1$). Then $F$ and $G$ differ with respect to the assigned $q$-weight (only) on partitions of the form $\pi=\alpha x \rho^*$, where $\alpha$ and $\rho^*$ are as described.  Such $\pi$ are enumerated by $x^{a-1}(F-1)$, since $\alpha x$ is non-empty and arbitrary and the $a-1$ appended letters comprising $\rho^*$ are determined once $\alpha$ is specified.  Subtracting the weight of such $\pi$ from the count for $G$, and adding them back with an extra factor of $q$, implies \eqref{thtau1be1}.

We now write a formula for $F$.  To do so, note that $\pi \in NC_n$ for some $n \geq 0$ may be expressed as (i) $\pi=1^n$, (ii) $\pi=1^r\alpha$, where $r \geq 1$ and $\alpha$ is nonempty and does not contain $1$, (iii) $\pi=1^r\alpha 1^s\beta$, where $0 \leq s \leq b-2$ and $\beta$ is nonempty and starts with exactly one $1$, (iv) $\pi=1^r\alpha 1^{b-1}\beta$, where $\beta$ is nonempty but may start with any positive number of $1$'s in this case. Note that the gf for all nonempty non-crossing partitions starting with a single $1$ according to the number of occurrences of $\tau$ is given by $F-1-x(F-1)=(1-x)(F-1)$, by subtraction.  Hence, case (iii) is seen to contribute $\frac{x}{1-x}(F-1)(1+x+\cdots+x^{b-2})(1-x)(F-1)$ towards $F$ if $b \geq 2$, with (iii) not applicable (i.e., it is subsumed by (iv)) if $b=1$.  In case (iv), one gets a contribution of $\frac{x}{1-x}(G-1)x^{b-1}(F-1)$ towards $F$ for all $b \geq 1$, where the $G-1$ factor accounts for the nonempty section $\alpha$, as it is followed by (at least) $b$ letters $1$.  Combining (i)--(iv) then gives
\begin{equation}\label{thtau1be2}
F=\frac{1}{1-x}+\frac{x}{1-x}(F-1)+\frac{x-x^b}{1-x}(F-1)^2+\frac{x^b}{1-x}(F-1)(G-1).
\end{equation}

To solve \eqref{thtau1be1} and \eqref{thtau1be2}, it is easier to consider $U=F-1$.  Then \eqref{thtau1be1} implies $G-1=(1+(q-1)x^{a-1})U$ and thus \eqref{thtau1be2} may be rewritten as
\begin{equation}\label{thtau1be3}
U=\frac{x}{1-x}\left(1+U+(1-x^{b-1})U^2+x^{b-1}(1+(q-1)x^{a-1})U^2\right).
\end{equation}
Solving for $U$ in  \eqref{thtau1be3} gives
$$U=\frac{1-2x-\sqrt{1-4x-4(q-1)x^{a+b}}}{2x(1+(q-1)x^{a+b-2})},$$
which implies the desired formula for $F=U+1$.
\end{proof}

Note that the formula for $F_\tau$ in Theorem \ref{thtau1b} depends only on the length of the subword $\tau$.  A bijective proof showing the equivalence of $\tau$ and $\tau'$ of the same length is given below.  In particular, when $|\tau|=3$, we have $211\equiv 221$ as subwords on $NC_n$ for all $n$, with the common gf formula given by
$$\frac{1+2(q-1)x^2-\sqrt{1-4x-4(q-1)x^3}}{2x(1+(q-1)x)}.$$

Differentiating the formula in Theorem \ref{thtau1b} gives the following.
\begin{corollary}\label{thtau1bc1}
If $n \geq a+b-1$, then the total number of occurrences of $\tau=(\rho+1)1^b$ as described above within all the members of $NC_n$ is given by $\binom{2r-2}{r+1}$, where $r=n-a-b+2$.
\end{corollary}

\emph{Remark:} For each $m \geq1$, we have from Corollary \ref{1a1b2cor2} that the nonzero values in the sequences for the total number of occurrences of $1^m$ and $1^m2$ in $NC_n$ for $n \geq 1$ correspond respectively to A001791 and A002054 in \cite{Sloane}. Corollary \ref{thtau1bc1} implies the total number of occurrences of $(\rho+1)1^b$ corresponds to A002694. \smallskip

Suppose $\tau=(p+1)1^b$ is as described above with $|\rho|=a$ and $\tau'=(\rho'+1)1^{b'}$, where $\rho'$ is of length $a' \geq 1$ and satisfies the same requirements as $\rho$ above, $b' \geq 1$ and $a'+b'=a+b$.\medskip

\textbf{Bijective proof of $\tau \equiv \tau'$ as subwords on $NC_n$:}\medskip

Clearly, we may assume $|\tau|=a+b \geq 3$.  We first prove the result when $b=b'=1$.  Let $\pi \in NC_n$, represented sequentially.  Let $\bf{s}$ denote a string of $\pi$ of the form ${\bf s}=u\alpha v$, where $\alpha \neq \varnothing$ and $1 \leq v<u\leq \min(\alpha)$.  If $\bf{s}$ corresponds to an occurrence $\tau$ ($\tau'$), then we will refer to $\bf{s}$ as an $\tau$-\emph{string} ($\tau'$-\emph{string}, respectively).  We wish to define a bijection $f$ on $NC_n$ in which partitions containing a given number of $\tau$-strings are mapped to those containing the same number of $\tau'$-strings, and vice versa.    If no $\tau$- or $\tau'$-strings exist (i.e, if $\pi$ avoids both $\tau$ and $\tau'$ as subwords), then let $f(\pi)=\pi$. So let $x_1,x_2,\ldots,x_r$ where $r \geq 1$ denote the complete combined set of $\tau$- and $\tau'$-strings in a left-to-right scan of the sequence $\pi$.  Note that since $b=b'=1$, the adjacent strings $x_i$ and $x_{i+1}$ for some $1 \leq i \leq r-1$ are either disjoint or share a single letter.

We now change each $x_i$ to the other option regarding containment of $\tau$ or $\tau'$.  We will first change $x_1$ and then subsequently work on $x_2,x_3,\ldots,x_r$, going from left to right.  Suppose first that $x_1$ is a $\tau$-string.  Then we will change $x_1$ to a $\tau'$-string $y_1$ as follows.  Similar reasoning will apply to the case when $x_1$ is a $\tau'$-string. Suppose $\tau$ has $s+1$ distinct letters, where $s \geq 1$, and that the $\tau$-string $x_1$ makes use of the actual letters $v<u=u_1<u_2<\cdots<u_s$.  Note that $\rho$ a partition and $\pi$ non-crossing implies $u_2,\ldots,u_s$ represent the leftmost occurrences of the letters of their respective kinds within $\pi$ and hence $u_\ell=u_2+\ell-2$ for $2 \leq \ell \leq s$.  Suppose $\tau'$ has $t+1$ distinct letters, where $t \geq 1$.  If $s \geq t$, then replace the letters in $x_1$ with a sequence that is isomorphic to $\tau'$ in which the roles of $1,2,\ldots,t+1$ are played by $v<u_1<\cdots<u_{t}$.  Further, if $s>t$, then the letters $u_{t+1}<\cdots <u_s$ are not needed in this replacement, in which case, we reduce each letter of $\pi$ belonging to $\{u_{s}+1,u_s+2,\ldots\}$, all of which must necessarily occur to the right of $x_1$ within $\pi$, by the amount $s-t$. Note that $\pi$ non-crossing and $\tau$ starting with $2$ and ending in $1$ implies that the letters $u_{t+1},\ldots,u_s$ within $x_1$ do not occur elsewhere in $\pi$.

On the other hand, if $t>s$, then we use all of the distinct letters occurring in $x_1$, together with $u_s+1,\ldots,u_s+t-s$, when performing the replacement. In this case, we must increase any letters of $\pi$ greater than or equal to $u_s+1$, all of which must occur to the right of $x_1$, by the amount $t-s$ in order to accommodate the new letters used.  In all cases, let $y_1$ denote the $\tau'$-string that results from making the replacement as described and let $\pi_1$ be the resulting member of $NC_n$.  Note that the combined set of $\tau$- and $\tau'$-strings in $\pi_1$ is given by $y_1,x_2,\ldots,x_r$.  We then repeat the process described above on $\pi_1$ in replacing $x_2$ with a string $y_2$ that represents the other option concerning containment of $\tau$ or $\tau'$, and let $\pi_2$ denote the resulting member of $NC_n$.  Likewise, we continue with $x_3,\ldots,x_r$, and convert them sequentially to $y_3,\ldots,y_r$, letting $\pi_3,\ldots,\pi_r$ denote the corresponding partitions that arise.

Let $f(\pi)=\pi_r$ and we show that $f$ can be reversed.  To do so, first note that the positions of the first and last letters of the strings $y_1,\ldots,y_r$ in $\pi_r$ are the same as the corresponding positions  within $x_1,\ldots,x_r$ in $\pi$, as they are seen to be invariant in each step of the transition from $\pi$ to $\pi_r$.  This follows from the fact that the first and last letters within an occurrence $z$ of $\tau$ or $\tau'$ are the two smallest letters in $z$.  Therefore, the inverse of $f$ may be found by reversing each of the transitions $\pi_i$ to $\pi_{i+1}$ for $0 \leq i \leq r-1$, where $\pi_0=\pi$, in reverse order (i.e., starting with the $i=r-1$ transition and ending with $i=0$).  Hence, we have $\mu_\tau(\pi)=\mu_{\tau'}(f(\pi))$ for all $\pi \in NC_n$ when it is assumed $b=b'=1$.

To complete the proof, it then suffices to show $2\sigma1^b\equiv 2^b\sigma 1$, where $b \geq 2$ and $2\sigma$ is a nonempty non-crossing partition (using the letters in $\{2,3,\ldots\}$) such that $\sigma$ starts with $3$ if nonempty. To establish this equivalence, let $\pi=\pi_1\cdots \pi_n \in NC_n$ and we consider (maximal) strings $\textbf{p}$ within $\pi$ of the form
$$\textbf{p}=u_1^{r_1}\sigma_1u_2^{r_2}\sigma_2\cdots u_t^{r_t}\sigma_tu_{t+1}^{r_{t+1}},$$
where $t,r_1,\ldots,r_t \geq 1$, $r_{t+1} \geq 0$, $u_1>u_2>\cdots>u_t$ (with $u_t>u_{t+1}$ if $r_{t+1}>0$ and $u_{t+1}=1$ if $r_{r+1}=\varnothing$) and $u_i\sigma_i$ isomorphic to $2\sigma$ for $1 \leq i \leq t$.  Note that if $r_{t+1}=0$, then either $u_t\sigma_t$ contains the last letter of $\pi$ or the successor of the final letter of $u_t\sigma_t$ is greater than or equal $u_t$ if $\sigma$ is nonempty (with the successor being strictly greater if $\sigma$ is empty).  Further, if $r_{t+1}>0$, then it is understood that $\sigma$ is nonempty and that the string $u_{t+1}^{r_{t+1}}$ is not directly followed by a sequence of letters $\alpha$ such that $u_{t+1}\alpha$ is isomorphic to $2\sigma$.  We replace each such string $\textbf{p}$ with $\textbf{p}'$, where
$$\textbf{p}'=\begin{cases}u_1^{r_{t+1}}\sigma_1u_2^{r_t}\sigma_2\cdots u_t^{r_2}\sigma_tu_{t+1}^{r_1}, \text{ if }r_{t+1}>0,\\u_1^{r_{t}}\sigma_1u_2^{r_{t-1}}\sigma_2\cdots u_t^{r_1}\sigma_tu_{t+1}^{r_{t+1}}, \text{ if } r_{t+1}=0. \end{cases}$$
Let $g(\pi)$ denote the member of $NC_n$ that results from replacing each string $\textbf{p}$ with $\textbf{p}'$ as described.  Then $g$ is an involution on $NC_n$ that replaces each occurrence of the pattern $2\sigma1^b$ with $2^b\sigma1$ and vice versa, which implies the desired equivalence and completes the proof.
\hfill \qed \medskip

\emph{Remarks:} When $|\tau|=|\tau'|=3$, then the bijection $f$ above shows $231 \equiv 221$.  For example, let $\pi=1\underline{231}1\underline{451}6\underline{78}\underline{\overline{6}}\overline{61}9 \in NC_{15}$, where the occurrences of 231 and 221 are underlined and overlined, respectively.  Then we have
\begin{align*}
\pi_0&\rightarrow \pi_1=1\overline{221}1\underline{341}5\underline{67}\underline{\overline{5}}\overline{51}8 \rightarrow \pi_2=1\overline{221}1\overline{331}4\underline{56}\underline{\overline{4}}\overline{41}7 \rightarrow \pi_3=1\overline{221}1\overline{331}4\overline{55441}6\\
&\rightarrow \pi_4=1\overline{221}1\overline{331}4\overline{55}\underline{\overline{4}}\underline{61}7,
\end{align*}
and thus $f(\pi)=\pi_4 \in NC_{15}$.  Note that $\pi$ has three occurrences of 231 and one of 221, whereas $f(\pi)$ has three occurrences of 221 and one of 231.  When $\tau$ and $\tau'$ are each of length three, the bijection $g$ shows $221 \equiv 211$. For example, if $n=12$ and $\pi=122322114115\in NC_{12}$, then  $g(\pi)=122332214415$. Note that $\pi$ and $g(\pi)$ contain one and three and three and one occurrences respectively of $221$ and $211$.  Finally, the mapping $g$ is seen to preserve the number of blocks of a partition, whereas $f$ does not in general. \smallskip

In the next result, we enumerate members of $NC_n$ with respect to a family of subword patterns generalizing  $121$.

\begin{theorem}\label{1arho1b}
Let $\tau=1^a(\rho+1)1^b$, where $a,b \geq 1$ and $\rho$ is the sequential representation of a non-crossing partition of length $m$ for some $m \geq 1$. Then the generating function counting the members of $NC_n$ for $n \geq0$ according to the number of occurrences of $\tau$ is given by
$$\frac{\left(1-x+(1-q)(1-x^s)x^{m+t}\right)\left(1-\sqrt{1-\frac{4x(1-x+(1-q)(1-x^{s-1})x^{m+t})}{1-x+(1-q)(1-x^s)x^{m+t}}}\right)}{2x\left(1-x+(1-q)(1-x^{s-1})x^{m+t}\right)},$$
where $s=\min\{a,b\}$ and $t=\max\{a,b\}$.
\end{theorem}
\begin{proof}
First assume $b \geq a>1$.  To find a formula for $F=F_\tau$ in this case, we refine $F$ by letting $F_i$ for $i \geq 1$ denote the restriction of $F$ to those partitions starting with a sequence of $1$'s of length exactly $i$.  Then we have $F_1=x+x(F-1)+x(F-1)^2=x(F^2-F+1)$, upon considering whether or not a partition enumerated by $F_1$ contains one or more runs of $1$.  By the definitions, we have $F_{i+1}=xF_i$ for all $i \neq a-1$, upon considering separately the cases $1 \leq i \leq a-2$ and $i \geq a$, since prepending an extra $1$ to a member of $NC_n$ not starting with a run of $1$ of length $a-1$ does not introduce an occurrence of $\tau$.  We now write a formula for $F_a$.  We consider the following cases on $\pi \in NC_n$ where $n \geq a$:  (i) $\pi=1^a\pi'$, where $\pi'$ contains no $1$'s and is possibly empty, (ii) $\pi=1^a\alpha\beta$, where $\alpha$ is nonempty and contains no $1$'s with $\alpha\neq \rho+1$ and $\beta$ is nonempty starting with $1$, (iii) $\pi=1^a\alpha\beta$, where $\alpha=\rho+1$ and $\beta$ is as before.  Note that $\beta$ in case (ii) is accounted for by $F-1$, whereas in (iii), we need
$$\sum_{i=1}^{b-1}F_i+q\sum_{i \geq b}F_i=\sum_{i=1}^{a-1}x^{i-1}F_1+\sum_{i=a}^{b-1}x^{i-a}F_a+q\sum_{i \geq b}x^{i-a}F_a=\frac{1-x^{a-1}}{1-x}F_1+\frac{1+(q-1)x^{b-a}}{1-x}F_a.$$

Thus, combining cases (i)--(iii), we have
$$F_a=x^aF+x^a(F-1-x^m)(F-1)+x^{m+a}\left(\frac{1-x^{a-1}}{1-x}F_1+\frac{1+(q-1)x^{b-a}}{1-x}F_a\right),$$
which implies
\begin{equation}\label{1arho1be1}
F_a=\frac{x^a(1+x^m)+x^a(F-1-x^m)F+\frac{x^{m+a}(1-x^{a-1})}{1-x}F_1}{1-\frac{x^{m+a}(1+(q-1)x^{b-a})}{1-x}}.
\end{equation}
We use the same cases (i)--(iii) in determining $F$ (except that the initial run of $1$'s can have arbitrary length in (i) and any length $\geq a$ in (ii) and (iii)), along with an additional case where $\pi$ is of the form $\pi=1^r\alpha\beta$, wherein $1 \leq r \leq a-1$ and $\alpha$ and $\beta$ are nonempty with $\alpha$ not containing $1$ and $\beta$ starting with $1$.  This yields
\begin{align}
F&=1+\frac{x}{1-x}F+\frac{x-x^a}{1-x}(F-1)^2+\frac{x^a}{1-x}(F-1-x^m)(F-1)\notag\\
&\quad+\frac{x^{m+a}}{1-x}\left(\sum_{i=1}^{b-1}F_i+q\sum_{i \geq b}F_i\right)\notag\\
&=1+\frac{x}{1-x}F+\frac{x-x^a}{1-x}(F-1)^2+\frac{x^a}{1-x}(F-1-x^m)(F-1)+\frac{x^{m+a}(1-x^{a-1})}{(1-x)^2}F_1\notag\\
&\quad+\frac{x^{m+a}(1+(q-1)x^{b-a})}{1-x}\cdot \frac{x^a(1+x^m)+x^a(F-1-x^m)F+\frac{x^{m+a}(1-x^{a-1})}{1-x}F_1}{1-x-x^{m+a}(1+(q-1)x^{b-a})},\label{1arho1be2}
\end{align}
where we have made use of \eqref{1arho1be1}.

Note that the $F_1$ coefficient in \eqref{1arho1be2} may be simplified to give
\begin{align*}
&\frac{x^{m+a}(1-x^{a-1})}{(1-x)^2}+\frac{x^{2(m+a)}(1-x^{a-1})(1+(q-1)x^{b-a})}{(1-x)^2(1-x-x^{m+a}(1+(q-1)x^{b-a}))}\\
&=\frac{x^{m+a}(1-x^{a-1})}{(1-x)^2}\left(1+\frac{x^{m+a}(1+(q-1)x^{b-a})}{1-x-x^{m+a}(1+(q-1)x^{b-a})}\right)\\
&=\frac{x^{m+a}(1-x^{a-1})}{(1-x)(1-x-x^{m+a}(1+(q-1)x^{b-a}))}.
\end{align*}
Thus, upon clearing fractions in \eqref{1arho1be2}, we have
\begin{align*}
(1-x-\ell)(F-1)&=x(1-x-\ell)F^2-x^{m+a}(1-x-\ell)(F-1)\\
&\quad+\ell x^a(1+x^m+(F-1-x^m)F)+x^{m+a}(1-x^{a-1})F_1,
\end{align*}
where $\ell=x^{m+a}(1+(q-1)x^{b-a})$.  By the formula for $F_1$, the last equation after several algebraic steps yields
$$(1-x+(1-q)(1-x^a)x^{m+b})(F-1)=x(1-x+(1-q)(1-x^{a-1})x^{m+b})F^2,$$
which leads to the stated formula for $F$ in this case.

Now let us consider the case $a=1$ and $b \geq 1$.  By similar reasoning as above, we have
\begin{align*}
F&=1+\frac{x}{1-x}F+\frac{x}{1-x}(F-1-x^m)(F-1)+\frac{x^{m+1}(1+(q-1)x^{b-1})}{(1-x)^2}F_1,\\
F_1&=xF+x(F-1-x^m)(F-1)+\frac{x^{m+1}(1+(q-1)x^{b-1})}{1-x}F_1.
\end{align*}
Solving this system for $F$ gives
$$F=\frac{1+(1-q)x^{m+b}-\sqrt{(1+(1-q)x^{m+b})(1-4x+(1-q)x^{m+b})}}{2x},$$
which establishes all cases of the formula when $b \geq a \geq 1$.

By a comparable argument, one can establish the stated formula for $F$ when $a>b \geq 1$.  Alternatively, note that the formula is symmetric in $a$ and $b$.  Thus, to complete the proof, it suffices to define a bijection on $NC_n$ showing that the $\mu_\tau$ statistic when $\tau=1^a(\rho+1)1^b$ has the same distribution as $\mu_{\tau'}$ for $\tau'=1^b(\rho+1)1^a$ where $a>b \geq 1$. By a \emph{maximal $\tau$-string} within $\pi=\pi_1\cdots \pi_n \in NC_n$, we mean a sequence ${\bf s}$ of consecutive letters of $\pi$ of the form ${\bf s}=x^{i_1}\alpha_1x^{i_2}\alpha_2\cdots x^{i_r}\alpha_rx^{i_{r+1}}$, where $r,i_1,\ldots,i_{r+1} \geq 1$, each $\alpha_i$ is isomorphic to $\rho$ and $x<\min\{\alpha_1\cup\alpha_2\cup\cdots\cup\alpha_r\}$, that is contained in no other such string of strictly greater length. Identify all maximal $\tau$-strings ${\bf s}$ within $\pi$; note that the various ${\bf s}$ are mutually disjoint, by maximality. Within each string, replace $x^{i_1},x^{i_2},\ldots,x^{i_{r+1}}$ with
$x^{i_{r+1}},x^{i_r},\ldots,x^{i_1}$ (i.e., reverse the order of the $x$-runs), leaving the $\alpha_i$ unchanged.  Let $\pi' \in NC_n$ denote the partition that results from performing this operation on all maximal $\tau$-strings ${\bf s}$; note that $\pi \mapsto \pi'$ is an involution and hence bijective.  Since any occurrence of $\tau$ must lie within some ${\bf s}$, the mapping $\pi \mapsto \pi'$ implies the desired equivalence of distributions and completes the proof.
\end{proof}

Theorem \ref{1arho1b} yields the following formula for the total number of occurrences of $1^a(\rho+1)1^b$.

\begin{corollary}\label{1arho1bcor1}
If $n \geq m+a+b-1$, then the total number of occurrences of $\tau=1^a(\rho+1)1^b$ as described above within all the members of $NC_n$ is given by $\binom{2r}{r+1}$, where $r=n-m-a-b+1$.
\end{corollary}

\noindent \emph{Remarks:} When $s=1$ in Theorem \ref{1arho1b}, the formula for $F=F_\tau$ may be simplified further to give
$$F=\frac{1+(1-q)x^{a+m}-\sqrt{(1+(1-q)x^{a+m})(1-4x+(1-q)x^{a+m})}}{2x},$$
where $\tau=1^a(\rho+1)1$ or $1(\rho+1)1^a$ and $a \geq 1$. Note that there is really no loss of generality in assuming $\rho$ is a sequential representation of some (non-crossing) partition in the hypotheses for Theorem \ref{1arho1b} above.  This is because if the first occurrence of some letter $c$ in $\rho$ precedes the first occurrence of $d$ with $c>d$, then containment of $\tau=1^a(\rho+1)1^b$ by a partition $\pi$ would imply an occurrence of 1-2-1-2 of the form $y$-$z$-$y$-$z$, where $z$ corresponds to the $d+1$ in $\rho+1$ and $y$ to the $1$ of $\tau$.  In addition to implying the symmetry in $a$ and $b$ of the pattern $\tau$, the formula in Theorem \ref{1arho1b} shows that $\tau=1^a(\rho+1)1^b$ is equivalent to $\tau'=1^a(\rho'+1)1^{b'}$ of  the same length, where $\rho'$ denotes a nonempty non-crossing partition and $a \leq \min\{b,b'\}$. For example, when $|\tau|=4$, we have $1121 \equiv 1211 \equiv 1221 \equiv 1231$ as subwords on $NC_n$.  A bijective proof of $1^a(\rho+1)1^b \equiv 1^a(\rho'+1)1^{b'}$ can be obtained by modifying somewhat the mapping $f$ described above, the details of which we leave to the interested reader.

\section{The subwords $12\cdots(m-1)m^a$ and $1^a23\cdots m$}

Let $\tau=12\cdots(m-1)m^a$, where $a,m \geq2$.  To aid in enumerating the members of $NC_n$ with respect to occurrences of $\tau$, we consider the joint distribution with a further parameter on $NC_n$ that was introduced in \cite{MSh}.  Given $\pi=\pi_1\cdots \pi_n \in NC_n$, excluding the increasing partition $12\cdots n$, let $\text{rep}(\pi)$ denote the smallest repeated letter of $\pi$.  Below, we will find, more generally, the gf for the joint distribution $\sum_{\pi \in NC_n}v^{\text{rep}(\pi)}q^{\mu_\tau(\pi)}$,
where $\text{rep}(12\cdots n)$ is defined to be zero.

Let $NC_{n,i}$ for $1 \leq i \leq n-1$ denote the subset of $NC_n$ whose members have smallest repeated letter $i$.  Define $a(n,i)=\sum_{\pi \in NC_{n,i}}q^{\mu_\tau(\pi)}$ for $n \geq 2$ and $1 \leq i \leq n-1$ and
$a(n)=\sum_{\pi \in NC_n}q^{\mu_\tau(\pi)}$ for $n \geq 1$, with $a(0)=1$.

To aid in finding recurrences for $a(n)$ and $a(n,i)$, we consider a generalization of $\mu_\tau$ as follows.  Given $\ell \geq 0$ and a partition $\pi$, let $\mu^{(\ell)}_\tau(\pi)$ denote the number of occurrences of $\tau$ in the sequence $12\cdots \ell(\pi+\ell)$.  Define
$$a^{(\ell)}(n,i)=\sum_{\pi \in NC_{n,i}}q^{\mu^{(\ell)}_\tau(\pi)}, \qquad n \geq 2 \text{ and } 1 \leq i \leq n-1,$$ and
$$a^{(\ell)}(n)=\sum_{\pi \in NC_n}q^{\mu^{(\ell)}_\tau(\pi)}, \qquad n \geq 1,$$
with $a^{(\ell)}(0)=1$.  Note that $a^{(0)}(n,i)=a(n,i)$ and $a^{(0)}(n)=a(n)$ for all $n$ and $i$.

We have the following system of recurrences satisfied by the $a^{(\ell)}(n,i)$ and $a^{(\ell)}(n)$.

\begin{lemma}\label{12m-1mal1}
If $n \geq a$ and $1 \leq i \leq n-a+1$, then
\begin{equation}\label{12m-1mal1e1}
a^{(\ell)}(n,i)=\sum_{j=i+1}^{n}a^{(\ell+i)}(j-i-1)a^{(0)}(n-j+1)+
\begin{cases}0, \text{ if }i+\ell \leq m-1,\\(q-1)a^{(0)}(n-i-a+2), \text{ if } i+\ell \geq m, \end{cases}
\end{equation}
for all $\ell \geq 0$.  Furthermore, we have
\begin{equation}\label{12m-1mal1e2}
a^{(\ell)}(n)=C_{a-1}+\sum_{i=1}^{n-a+1}a^{(\ell)}(n,i), \qquad n \geq a,
\end{equation}
with $a^{(\ell)}(n)=C_n$ for $0 \leq n \leq a-1$.
\end{lemma}
\begin{proof}
Since $a^{(\ell)}(0)=1$ for all $\ell \geq 0$, formula \eqref{12m-1mal1e1} is equivalent to
\footnotesize{\begin{align}\label{12m-1mal1e1a}
a^{\ell}(n,i)=\sum_{j=i+2}^{n}a^{(\ell+i)}(j-i-1)a^{(0)}(n-j+1)+
\begin{cases}a^{(0)}(n-i), \text{ if }i+\ell \leq m-1,\\ a^{(0)}(n-i)+(q-1)a^{(0)}(n-i-a+2), \text{ if } i+\ell \geq m, \end{cases}\end{align}}\normalsize
which we will now show. To do so, first consider the position $j$ of the second occurrence of $i$ within $\pi \in NC_{n,i}$.  If $j \geq i+2$, such $\pi$ are expressible as $\pi=12\cdots i\alpha i\beta$, where $\alpha$ is nonempty and contains no letters $i$ and $\beta$ is possibly empty. Then we get $a^{(\ell+i)}(j-i-1)a^{(0)}(n-j+1)$ possibilities and summing over all $j \geq i+2$ yields the first part of \eqref{12m-1mal1e1a} in either case.  So assume $j=i+1$ and first suppose $i+\ell \leq m-1$.  Then there is no occurrence of $\tau$ in $\pi$ involving any of its first $i-1$ letters, regardless of the length of the leftmost run of $i$'s, which implies a contribution of $a^{(0)}(n-i)$ and hence the first case of \eqref{12m-1mal1e1a}.  If $i+\ell \geq m$, then we consider cases based on the length of the leftmost run of $i$'s as follows.  Suppose first that $\pi$ is expressible as $\pi=12\cdots(i-1)i^{r}\pi'$, where $\pi'$ does not start with $i$ and $2 \leq r \leq a-1$, assuming for now $a \geq 3$.  Then, by subtraction, there are $a^{(0)}(n-i-r+2)-a^{(0)}(n-i-r+1)$ possibilities and summing over all $r$ gives
$$\sum_{r=2}^{a-1}(a^{(0)}(n-i-r+2)-a^{(0)}(n-i-r+1))=a^{(0)}(n-i)-a^{(0)}(n-i-a+2).$$
On the other hand, if $\pi=12\cdots(i-1)i^r\pi'$, where $r \geq a$, then $i+\ell \geq m$ implies that there is an occurrence of $\tau$ involving the first $i+a-1$ letters of $\pi$ (when taken together with the understood suffix $12\cdots \ell$ consisting of strictly smaller letters).  Then the sequence $i^{r-a+1}\pi'$ corresponds to a partition enumerated by $a^{(0)}(n-i-a+2)$, as it is directly preceded by at least one $i$, and hence the contribution towards the overall weight in this case is given by $qa^{(0)}(n-i-a+2)$.  Combining this case with the previous yields the second part of \eqref{12m-1mal1e1a} when $i+\ell \geq m$ and completes the proof of \eqref{12m-1mal1e1a}.

For \eqref{12m-1mal1e2}, first note that the initial conditions when $0 \leq n \leq a-1$ are apparent since no occurrence of $\tau$ is possible for such $n$ for all $\ell$.  Suppose $k$ is the smallest repeated letter in $\pi \in NC_n$.  If $1 \leq k \leq n-a+1$, then $\pi$ is accounted for by the sum in \eqref{12m-1mal1e2}, by the definitions.  Otherwise, $\pi$ can be represented as $\pi=12\cdots (n-a+1)\pi'$, where $\pi'$ contains no letters in $[n-a+1]$, for which there are $C_{a-1}$ possibilities since no such $\pi$ can contain an occurrence of $\tau$ (as the $m^a$ part of $\tau$ cannot be achieved by any letter in $\pi'$).  Combining this with the prior case yields \eqref{12m-1mal1e2}.
\end{proof}

Define
$$A(x,u)=\sum_{n \geq0}\sum_{\ell \geq0}a^{(\ell)}(n)u^\ell x^n$$
and
$$A(x,u,v)=\sum_{n \geq a}\sum_{\ell \geq0}\sum_{i=1}^{n-a+1}a^{(\ell)}(n,i)u^\ell v^{i-1} x^n.$$
Rewriting the recurrences in Lemma \ref{12m-1mal1} in terms of gf's yields the following system of functional equations.

\begin{lemma}\label{12m-1mal2}
We have
\begin{equation}\label{12m-1mal2e1}
A(x,u)=A(x,u,1)+\frac{x^aC_{a-1}}{(1-u)(1-x)}+L(x,u),
\end{equation}
\begin{align}
A(x,u,v)&=\frac{x(A(x,0)-1)(A(x,vx)-A(x,u))}{vx-u}-M(x,u,v)\notag\\
&\quad+\frac{(q-1)x^{a-1}(A(x,0)-1)(u^m(1-vx)-(vx)^m(1-u))}{(1-u)(1-vx)(u-vx)},\label{12m-1mal2e2}
\end{align}
where $L(x,u)=\frac{1}{1-u}\sum_{j=0}^{a-1}C_jx^j$ and $$M(x,u,v)=\begin{cases} 0, \text{ if } a=2,\\ \frac{x}{(1-u)(1-vx)}\sum_{j=0}^{a-3}\sum_{i=1}^{a-2-j}C_iC_jx^{i+j}, \text{ if } a \geq 3. \end{cases}$$
\end{lemma}
\begin{proof}
Multiplying both sides of \eqref{12m-1mal1e2} by $u^\ell x^n$, and summing over $n \geq a$ and $\ell \geq 0$, gives
\begin{align*}
A(x,u)&=\sum_{n\geq a}\sum_{\ell \geq0}\sum_{i=1}^{n-a+1}a^{(\ell)}(n,i)u^\ell x^n+\sum_{n\geq a}\sum_{\ell \geq0}C_{a-1}u^\ell x^n+\sum_{n=0}^{a-1}\sum_{\ell \geq0}a^{(\ell)}(n)u^\ell x^n\\
&=A(x,u,1)+\frac{x^aC_{a-1}}{(1-u)(1-x)}+\frac{1}{1-u}\sum_{j=0}^{a-1}C_jx^j,
\end{align*}
by the initial values for $a^{(\ell)}(n)$.

To rewrite \eqref{12m-1mal1e1} in terms of gf's, we first must find
$$\sum_{n \geq a}\sum_{\ell \geq 0}\sum_{i=1}^{n-a+1}\sum_{j=i+1}^na^{(\ell+i)}(j-i-1)a^{(0)}(n-j+1)u^\ell v^{i-1}x^n.$$
First observe the following manipulation of sums:
\begin{align*}
&\sum_{n \geq a}\sum_{\ell \geq 0}\sum_{i=1}^{n-a+1}\sum_{j=i+1}^n(
\ldots)=\sum_{\ell \geq0}\sum_{i \geq 1}\sum_{j \geq i+1}\sum_{n \geq \max\{j,i+a-1\}}(\ldots)\\
&=\sum_{\ell \geq0}\sum_{i \geq 1}\sum_{j=i+1}^{i+a-1}\sum_{n \geq i+a-1}(\ldots) + \sum_{\ell \geq0}\sum_{i \geq 1}\sum_{j \geq i+a}\sum_{n \geq j}(\ldots),
\end{align*}
where $(\ldots)$ denotes the original summand above.  Replacing $j$ with $j+i+1$ in both sums in the last expression implies
\begin{align*}
&\sum_{n \geq a}\sum_{\ell \geq 0}\sum_{i=1}^{n-a+1}\sum_{j=i+1}^na^{(\ell+i)}(j-i-1)a^{(0)}(n-j+1)u^\ell v^{i-1}x^n\\
&=\sum_{\ell \geq0}\sum_{i \geq 1}\sum_{j=0}^{a-2}\sum_{n \geq i+a-1}a^{(\ell+i)}(j)a^{(0)}(n-j-i)u^\ell v^{i-1}x^n\\
&\quad+\sum_{\ell \geq0}\sum_{i \geq 1}\sum_{j\geq a-1}\sum_{n \geq j+i+1}a^{(\ell+i)}(j)a^{(0)}(n-j-i)u^\ell v^{i-1}x^n\\
&=\sum_{\ell \geq0}\sum_{i \geq 1}\sum_{j=0}^{a-2}\sum_{n \geq a-1-j}a^{(\ell+i)}(j)a^{(0)}(n)u^\ell v^{i-1}x^{n+i+j}\\
&\quad+\sum_{\ell \geq0}\sum_{i \geq 1}\sum_{j\geq a-1}\sum_{n\geq1}a^{(\ell+i)}(j)a^{(0)}(n)u^\ell v^{i-1}x^{n+i+j}\\
&=\sum_{\ell \geq0}\sum_{i \geq 1}\sum_{j=0}^{a-2}a^{(\ell+i)}(j)u^\ell v^{i-1}x^{i+j}\left(\sum_{n\geq1}a^{(0)}(n)x^n-\sum_{n=1}^{a-2-j}a^{(0)}(n)x^n\right)\\
&\quad+\sum_{\ell \geq0}\sum_{i \geq 1}\sum_{j\geq a-1}a^{(\ell+i)}(j)u^\ell v^{i-1}x^{i+j}\sum_{n\geq1}a^{(0)}(n)x^n\\
&=\sum_{\ell \geq0}\sum_{i \geq 1}\sum_{j\geq 0}a^{(\ell+i)}(j)u^\ell v^{i-1}x^{i+j}\sum_{n \geq 1}a^{(0)}(n)x^n\\
&\quad-\sum_{\ell \geq0}\sum_{i \geq 1}\sum_{j=0}^{a-2}a^{(\ell+i)}(j)u^\ell v^{i-1}x^{i+j}\sum_{n=1}^{a-2-j}a^{(0)}(n)x^n\\
&=(A(x,0)-1)\sum_{i\geq1}\sum_{\ell \geq i}\sum_{j \geq0}a^{(\ell)}(j)u^{\ell-i}v^{i-1}x^{i+j}-\frac{x}{(1-u)(1-vx)}\sum_{j=0}^{a-3}C_jx^j\sum_{n=1}^{a-2-j}C_nx^n\\
&=\frac{x(A(x,0)-1)}{u-vx}\sum_{\ell \geq 1}\sum_{j \geq 0}a^{(\ell)}(j)x^j(u^\ell-(vx)^\ell)-M(x,u,v)\\
&=\frac{x(A(x,0)-1)(A(x,vx)-A(x,u))}{vx-u}-M(x,u,v).
\end{align*}

For converting the second part of formula \eqref{12m-1mal1e1}, we consider cases on $i$.  Omitting the factor $q-1$, this yields
\begin{align*}
&\sum_{i=1}^{m-1}\sum_{\ell \geq m-i}\sum_{n \geq i+a-1}a^{(0)}(n-i-a+2)u^\ell v^{i-1}x^n+\sum_{i \geq m}\sum_{\ell \geq0}\sum_{n \geq i+a-1}a^{(0)}(n-i-a+2)u^\ell v^{i-1}x^n\\
&=(A(x,0)-1)\sum_{i=1}^{m-1}\sum_{\ell \geq m-i}u^\ell v^{i-1}x^{i+a-2}+(A(x,0)-1)\sum_{i\geq m}\sum_{\ell \geq 0}u^\ell v^{i-1}x^{i+a-2}\\
&=(A(x,0)-1)\left(\frac{ux^{a-1}(u^{m-1}-(vx)^{m-1})}{(1-u)(u-vx)}+\frac{v^{m-1}x^{a+m-2}}{(1-u)(1-vx)}\right)\\
&=\frac{x^{a-1}(A(x,0)-1)(u^m(1-vx)-(vx)^m(1-u))}{(1-u)(1-vx)(u-vx)}.
\end{align*}
Combining the two contributions to the gf above yields \eqref{12m-1mal2e2}.
\end{proof}

\begin{theorem}\label{t1m-1mal2}
Let $y=A(x,0)$ denote the generating function counting members of $NC_n$ for $n \geq 0$ according to the number of occurrences of $\tau=12\cdots(m-1)m^a$, where $a,m \geq 2$.  Then $y$ satisfies the polynomial equation
\begin{equation}\label{t1m-1mal2e1}
xy^2-y+1+(q-1)x^{a+m-2}y^{m-1}(y-1)=0.
\end{equation}
More generally, the generating function counting members of $NC_n$ jointly according to the smallest repeated letter and number of occurrences of $\tau$ (marked by $v$ and $q$, respectively) is given by $vA(x,0,v)+\frac{1}{1-x}$ if $a=2$ and by
$$vA(x,0,v)+\frac{1}{1-x}+\frac{v^2x^aC_{a-1}-v^2x^2}{1-vx}+v\sum_{i=2}^{a-1}C_ix^i+(v-1)\sum_{n\geq 2}\sum_{j=r}^{n-1}C_{n-j}x^j,$$
if $a \geq 3$, where $r=\max\{1,n-a+2\}$ and
\begin{equation}\label{t1m-1mal2e2}
A(x,0,v)=\frac{(y-1)(A(x,vx)-y)}{v}-M(x,0,v)+\frac{(q-1)v^{m-1}x^{a+m-2}(y-1)}{1-vx},
\end{equation}
with  $A(x,u)$ given by \eqref{t1m-1mal2e3a}.
\end{theorem}
\begin{proof}
We first find an equation satisfied by $y$.  Note that \eqref{12m-1mal2e2} at $u=0$ and $v=1$, taken together with \eqref{12m-1mal2e1}, gives
\begin{equation}\label{t1m-1mal2e1a}
y^2=(y-1)A(x,x)-M(x,0,1)+\frac{(q-1)x^{a+m-2}(y-1)}{1-x}+\frac{x^aC_{a-1}}{1-x}+L(x,0).
\end{equation}
We apply the kernel method to \eqref{12m-1mal2e2} to obtain an expression for $A(x,x)$.  Taking $u=xA(x,0)=xy$ and $v=1$ in \eqref{12m-1mal2e2} implies
\begin{equation}\label{t1m-1mal2e2a}
A(x,x)=\frac{(q-1)x^{a-2}((xy)^m(1-x)-x^m(1-xy))}{(1-x)(1-xy)}-M(x,xy,1)+\frac{x^aC_{a-1}}{(1-x)(1-xy)}+L(x,xy).
\end{equation}
Now observe
$$\sum_{j=0}^{a-3}\sum_{i=1}^{a-2-j}C_iC_jx^{i+j}=\sum_{i=1}^{a-2}x^i\sum_{j=0}^{i-1}C_{i-j}C_j=\sum_{i=0}^{a-2}x^i(C_{i+1}-C_i),$$
by the recurrence for Catalan numbers.  Thus, the right-hand side of \eqref{t1m-1mal2e1a} may be written as
\begin{align*}
&\frac{(q-1)x^{a-2}((xy)^m(1-x)-x^m(1-xy))(y-1)}{(1-x)(1-xy)}-\frac{x(y-1)}{(1-x)(1-xy)}\sum_{i=0}^{a-2}x^i(C_{i+1}-C_i)\\
&\quad+(y-1)\left(\frac{x^aC_{a-1}}{(1-x)(1-xy)}+\frac{1}{1-xy}\sum_{j=0}^{a-1}C_jx^j\right)-\frac{x}{1-x}\sum_{i=0}^{a-2}x^i(C_{i+1}-C_i)\\
&\quad+\frac{(q-1)x^{a+m-2}(y-1)}{1-x}+\frac{x^aC_{a-1}}{1-x}+\sum_{j=0}^{a-1}C_jx^j\\
&=\frac{(q-1)x^{a+m-2}y^m(y-1)}{1-xy}-\frac{xy}{1-xy}\sum_{i=0}^{a-2}x^i(C_{i+1}-C_i)+\frac{x^ayC_{a-1}}{1-xy}+\frac{y(1-x)}{1-xy}\sum_{j=0}^{a-1}C_jx^j\\
&=\frac{y}{1-xy}\biggl((q-1)x^{a+m-2}y^{m-1}(y-1)-\sum_{i=1}^{a-1}C_ix^i+x\sum_{i=0}^{a-2}C_ix^i+x^aC_{a-1}\\
&\quad+(1-x)\sum_{j=0}^{a-1}C_jx^j\biggr)\\
&=\frac{(q-1)x^{a+m-2}y^m(y-1)+y}{1-xy}.
\end{align*}

Equating this last expression with $y^2$ then leads to \eqref{t1m-1mal2e1}.  Solving for $A(x,u)$ in \eqref{12m-1mal2e2} at $v=1$, making use of \eqref{12m-1mal2e1}, gives
\begin{align}
A(x,u)&=\frac{x-u}{xy-u}\biggr(\frac{x^aC_{a-1}}{(1-u)(1-x)}+\frac{(1-q)x^{a-1}(u^m(1-x)-x^m(1-u))(y-1)}{(1-u)(1-x)(x-u)}\notag\\
&\quad+\frac{x(y-1)}{x-u}A(x,x)-M(x,u,1)+L(x,u)\biggl),\label{t1m-1mal2e3a}
\end{align}
where $A(x,x)$ is given by \eqref{t1m-1mal2e2a}.  Letting $u=0$ in \eqref{12m-1mal2e2} now leads to \eqref{t1m-1mal2e2}.  Finally, taking into account the $v$-weights of members of $NC_n$ having smallest repeated letter $i$ where $r \leq i \leq n-1$, along with the increasing partition (which has weight $1$ for all $n \geq 0$), implies the gf enumerating members of $NC_n$ for $n \geq 0$ jointly according to the rep value and number of occurrences of $\tau$ is given by $vA(x,0,v)+\frac{1}{1-x}$ if $a=2$ and by
$$vA(x,0,v)+\frac{1}{1-x}+\sum_{n=2}^{a-1}x^n\sum_{j=1}^nv^j(C_{n-j+1}-C_{n-j})+\sum_{n\geq a}x^n\sum_{k=n-a+2}^nv^k(C_{n-k+1}-C_{n-k}),$$
if $a \geq 3$.  Rewriting the last expression somewhat yields the stated formula for the joint gf and completes the proof.
\end{proof}

\begin{corollary}\label{t1m-1mal2c1}
If $n \geq a+m-1$, then the total number of occurrences of $\tau=1\cdots(m-1)m^a$ within all the members of $NC_n$ is given by $\frac{r}{2r+m}\binom{2r+m}{r}$, where $r=n-a-m+2$.
\end{corollary}
\begin{proof}
Let $C=C(x)$, $F=A(x,0)$ and $D=\frac{\partial F}{\partial q}\mid_{q=1}$. Differentiating both sides of \eqref{t1m-1mal2e1} with respect to $q$, and noting $F\mid_{q=1}=C$, yields
$$2xCD-D+x^{a+m-2}C^{m-1}(C-1)=0,$$
i.e.,
$$D=\frac{x^{a+m-2}C^{m-1}(C-1)}{1-2xC}=\frac{x^{a+m-2}C^{m-1}(C-1)}{\sqrt{1-4x}}.$$
Extracting the coefficient of $x^n$ for $n \geq a+m-1$, and making use of \cite[Eqn.~2.5.15]{Wilf}, then gives
$$[x^n]D=\binom{2r+m}{r}-\binom{2r+m-1}{r}=\left(1-\frac{r+m}{2r+m}\right)\binom{2r+m}{r}=\frac{r}{2r+m}\binom{2r+m}{r}.$$
\end{proof}

\emph{Remark:}  The $m=2$ and $m=3$ cases of the formula $\frac{r}{2r+m}\binom{2r+m}{r}$ from Corollary \ref{t1m-1mal2c1} coincide respectively with sequences A002054 and A002694 in \cite{Sloane}. \smallskip

We conclude with the following equivalence between $\tau$ and $1^a23\cdots m$.

\begin{theorem}\label{1a2m12maequiv}
We have $1^a23\cdots m \equiv 12\cdots(m-1)m^a$ as subwords on $NC_n$ for all $a,m \geq 2$.
\end{theorem}
\begin{proof}
We provide a bijective proof of this result.  Suppose that the descents from left to right within $\pi=\pi_1\cdots \pi_n \in NC_n$ correspond to the letters $a_i>b_i$ for $1 \leq i \leq r$ and some $r\geq 0$.  Let $\rho_1$ denote the section of $\pi$ to the left of and including $a_1$ and $\rho_{r+1}$ the section to the right of and including $b_r$ (if $r=0$, then $\rho_1$ comprises all of $\pi$).  If $r \geq 2$, then let $\rho_i$ for $2 \leq i \leq r$ denote the subsequence of $\pi$ starting with $b_{i-1}$ and ending with $a_i$.  Note that $\rho_i$ for each $i$ is weakly increasing, as it consists of the letters between consecutive descents of $\pi$ (or occurring prior to the first or after the last descent of $\pi$).

Suppose that the descent bottom letters $b_1, \ldots ,b_r$ within $\pi$ are given, with $b_0=1$.  Let section $\rho_i$ of $\pi$ for $1 \leq i \leq r+1$ be represented sequentially as $\rho_i=s_0^{(i)}s_1^{(i)}\cdots s_{t_{i}}^{(i)}$, where $s_0^{(i)}=b_{i-1}$.  Define the binary sequence ${\bf d^{(i)}}=d_1^{(i)}d_2^{(i)}\cdots d_{t_i}^{(i)}$, where $d_k^{(i)}=1$ if $s_k^{(i)}>s_{k-1}^{(i)}$ and $d_k^{(i)}=0$ if $s_k^{(i)}=s_{k-1}^{(i)}$ for $1 \leq k \leq t_i$.  Note that $\pi$ non-crossing implies it is uniquely determined by its descent bottoms $b_1,\ldots,b_r$, taken together with its complete set of associated binary sequences ${\bf d^{(1)}},\ldots,{\bf d^{(r+1)}}$.

Let $\pi'$ be the uniquely determined member of $NC_n$ whose descents bottoms are the same as those of $\pi$ (i.e., are given by  $b_1,\ldots,b_{r}$) and whose associated binary sequences are given by $\text{rev}({\bf d^{(1)}}),\ldots,\text{rev}({\bf d^{(r+1)}})$, where $\text{rev}(s)$ denotes the reversal of a sequence $s$.  Note that the section $\rho_i'$ of $\pi'$ corresponding to $\rho_i$ in $\pi$ for $1 \leq i \leq r+1$ will have the same set of distinct letters for all $i$, and thus $\pi'$ will have the same ascent tops as $\pi$.  Also, an occurrence of $1^a23\cdots m$ or $12\cdots(m-1)m^a$ within some section $\rho_i$ of $\pi$ will result in an occurrence of the other pattern within $\rho_i'$ of $\pi'$, and vice versa.  Further, an occurrence of either pattern must lie completely within a section $\rho_i$ of $\pi$ or $\rho_i'$ of $\pi'$, as neither contains a descent.  Since the mapping $\pi \mapsto \pi'$ is an involution on $NC_n$, and hence bijective, the desired equivalence of patterns follows.
\end{proof}

\end{document}